\documentclass[oneside, 11pt, reqno]{amsart}

\usepackage[a4paper, margin=1in]{geometry}

\usepackage[utf8]{inputenc}
\usepackage[T1]{fontenc}
\usepackage{lmodern}

\usepackage{amssymb}
\usepackage{float}
\usepackage{mathrsfs}
\usepackage[dvipsnames]{xcolor}
\usepackage[shortlabels]{enumitem}
\usepackage{xspace}

\usepackage{tikz}
\usetikzlibrary{positioning}
\usetikzlibrary{shapes.geometric}

\usepackage{thmtools}
\usepackage{hyperref}
\hypersetup{linktoc=all,colorlinks=true,allcolors=Blue,breaklinks=true}
\usepackage[capitalize,noabbrev]{cleveref}

\newtheorem{lemma}{Lemma}[section]
\newtheorem{theorem}[lemma]{Theorem}
\newtheorem{proposition}[lemma]{Proposition}
\newtheorem{corollary}[lemma]{Corollary}
\newtheorem{claim}[lemma]{Claim}
\theoremstyle{definition}
\newtheorem{definition}[lemma]{Definition}
\newtheorem{remark}[lemma]{Remark}
\newtheorem{example}[lemma]{Example}
\newtheoremstyle{named}{}{}{\itshape}{}{\bfseries}{.}{.5em}{#3}
\theoremstyle{named}
\newtheorem*{theorem*}{Theorem}

\setlist[enumerate,1]{label={\upshape(\alph*)},
ref={\upshape\thetheorem(\alph*)}}% Default enumerate should be upshape.
\setlist[enumerate,2]{label={\upshape(\roman*)},
ref={\upshape\thetheorem(\roman*)}}% Default enumerate should be upshape.

\makeatletter
\newcommand{\itemlabel}[1]{%
  \ifnum\@enumdepth=1
    \protected@edef\@currentlabel{{(\alph{enumi})}}%
  \else\ifnum\@enumdepth=2
    \protected@edef\@currentlabel{{(\roman{enumii})}}%
  \fi\fi
  \oldlabel{#1}%
}

\newcommand\crefenumtype{} % Add cref key to lists
\SetEnumitemKey{cref}{%
  before={\def\crefenumtype{#1}%
  \let\oldlabel\label
  \renewcommand{\label}[2][]{%
    \ifx\crefenumtype\empty
      \oldlabel{####2}%
    \else
      \oldlabel[\crefenumtype]{####2}%
      \itemlabel{####2@item}%
    \fi}%
  }%
}
\newcommand{\itemref}[1]{\ref{#1@item}}
\makeatother

\DeclareMathOperator{\cl}{cl}
\DeclareMathOperator{\intr}{int}

\title{When is double negation Scott continuous?}

\date{}


\author{G. Bezhanishvili}
\address{Department of Mathematical Sciences, New Mexico State University, Las Cruces, NM, USA}
\email{guram@nmsu.edu}
\author{S. D. Melzer}
\address{Department of Mathematics, University of Salerno, Fisciano, Italy}
\email{smelzer@unisa.it}

\begin{document}
\begingroup
\def\uppercasenonmath#1{}
\let\MakeUppercase\relax
\maketitle
\endgroup

\begin{abstract}
Let $L$ be the frame of opens of a $T_0$-space $X$.
We prove that if $X$ is sober and $T_1$, then the double negation nucleus on $L$ is Scott continuous iff $X$ is discrete. 
It follows that if, in addition, $X$ is
compact then double negation is Scott continuous iff $X$ is finite. We show that both the sober and $T_1$ assumptions 
are essential, and generalize the above results to all boolean nuclei on $L$. A pointfree characterization of when $X$ is sober and $T_1$ is also given by proving that it is equivalent to Scott continuous nuclei on $L$ being closed. 
\end{abstract}

\vspace{1em}

\section{Introduction and motivation}

A \emph{nucleus} on a frame $L$ is a map $j:L\to L$ such that
\[
    a\leq ja,\qquad jja=ja,\qquad j(a\wedge b)=ja\wedge jb
\]
for all $a,b\in L$. 
Ordering nuclei pointwise yields the \emph{assembly frame} $\mathrm N L$ of $L$. If $j\in\mathrm N L$, then 
\[
    jL=\{a\in L\mid ja=a\}
\]
is a sublocale of $L$, and every sublocale arises this way. Thus, the assembly frame is dually isomorphic to the coframe of sublocales of $L$. As such, nuclei play a central role in pointfree topology (see, e.g., \cite{PP12}).

One of the most studied nuclei is the double negation nucleus, given by $a \mapsto a^{**}$, where $a^*$ denotes the pseudocomplement of $a$ in $L$. The sublocale of fixpoints of double negation is the boolean frame $\mathfrak B L$ known as the \emph{booleanization} of $L$. It is a classic result of Isbell \cite{Isb72} that $\mathfrak B L$ is the least dense sublocale of $L$, where we recall that a sublocale is \emph{dense} provided it contains the bottom element $0$ of $L$. 

We call a sublocale {\em boolean} if it is a boolean frame. It is well known (see, e.g., \cite[10.4]{PP12}) that a sublocale is boolean iff it is the fixpoints of a \emph{boolean nucleus}
\[
    \mathfrak b_u(a) = (a \to u) \to u
\]
for some $u \in L$. Boolean nuclei/sublocales play an important role since every nucleus is a meet of boolean nuclei \cite[Lem.~7(ii)]{Sim78}, and hence every sublocale is a join of boolean sublocales. 

In this note we are interested in nuclei that preserve directed joins: 

\begin{definition}\label{def: Scott continuous}
    A nucleus $j$ on a frame $L$ is \emph{Scott continuous} provided $j(\bigvee S)=\bigvee j[S]$ for each directed $S\subseteq L$. 
    We also call a sublocale \emph{Scott continuous} when its associated nucleus is Scott continuous.
\end{definition}

Scott continuous nuclei have been studied in the literature under various names, including \emph{finitary} nuclei \cite{Ban88,Sun88} and \emph{perfect} nuclei \cite{Esc99} (the latter paper also uses the terminology of \cref{def: Scott continuous}). 
If $L$ is an algebraic frame, then Scott continuous nuclei are precisely the \emph{inductive} nuclei of \cite{MZ03}.  

\begin{lemma}[{\cite[Lem.~3.1.8]{Esc98}}]
    Let $L$ be a frame. The set $\mathrm{SN} L$ of Scott continuous nuclei is a subframe of $\mathrm N L$.
    \label{lem:SNL-subframe}
\end{lemma}

As was observed in \cite{Esc99}, $\mathrm{SN} L$ serves as the pointfree analog of the patch topology, which plays a key role in the study of spectral spaces or, more generally, of stably locally compact spaces. 
As a consequence of \cref{lem:SNL-subframe}, we obtain that for each nucleus there is a largest Scott continuous nucleus beneath it. In particular, there exists a largest Scott continuous nucleus beneath the double negation nucleus. 
In the setting of arithmetic frames,\footnote{That is, algebraic frames in which the meet of two compact elements is again compact.} this is precisely the \emph{$d$-nucleus} of \cite{MZ03} defined by \[
    da = \bigvee \{k \in \mathrm K L \mid k \leq a\},
\]
where $\mathrm K L$ denotes the set of compact elements of $L$.
Thus, if $L$ is arithmetic, then 
double negation is Scott continuous iff it is equal to $d$. This motivates the question of when double negation is Scott continuous for an arbitrary $L$, which is the topic of this short note. 
More generally, we consider the question of when a boolean nucleus is Scott continuous. Our findings are summarized below, from which it follows that this happens rarely.  

We use the Hofmann-Mislove Theorem \cite{HM81} to prove that, on the frame of opens of a sober $T_1$-space, double negation is Scott continuous iff the space is discrete. We also show that neither 
hypothesis can be omitted, and that an analogous pointfree 
statement, that double negation is Scott continuous iff the frame is boolean, fails for both subfit and Hausdorff frames. For fit frames, Sexton and Simmons \cite[Thm.~10.6]{SS06} proved that every Scott continuous nucleus is closed, which implies the above statement.
We show that every Scott continuous nucleus on the frame of opens of a $T_0$-space $X$ is closed iff $X$ is sober and $T_1$, thus providing a pointfree characterization of sober $T_1$-spaces. We also generalize the characterization of when double negation is Scott continuous to arbitrary boolean nuclei on the frame of opens $\Omega(X)$ of a sober $T_1$-space $X$ by showing that the boolean nucleus associated with $U\in\Omega(X)$ is Scott continuous iff $X\setminus U$ is discrete. Finally, we generalize the notion of a unit (see \cite{KM07}) to that of a $j$-unit, where $j$ is a nucleus, and use this new notion to characterize Scott continuity of $j$ in the 
compact case. In particular, we prove that every boolean nucleus on a compact frame is Scott continuous iff the corresponding boolean sublocale is finite.

\section{Scott continuity for double negation}

We start by showing that the double negation nucleus on the frame of opens of a sober $T_1$-space is Scott continuous iff the space is discrete. To do so, we utilize the Hofmann-Mislove Theorem. 

Recall that a filter $F$ of a frame $L$ is \emph{Scott open} provided for all directed $S\subseteq L$, if $\bigvee S \in F$ then $F \cap S \neq \varnothing$.
It is known (see, e.g., \cite[Prop.~II-2.1]{GH+03})  that a Scott continuous nucleus pulls back
Scott open filters to Scott open filters. The next theorem shows that for
spatial frames the converse also holds, and that it is enough to test this on
completely prime filters.

\begin{theorem} \label{thm:spatial}
    Let $L$ be a spatial frame and $j$ a nucleus on $L$. The following are equivalent.
    \begin{enumerate}[cref=theorem]
        \item $j$ is Scott continuous;\label{spatial-cond-1}
        \item If $F \subseteq L$ is a Scott open filter, then $j^{-1}(F)$ is a Scott open filter.\label{spatial-cond-2}
        \item If $P \subseteq L$ is a completely prime filter, then $j^{-1}(P)$ is a Scott open filter.\label{spatial-cond-3}
    \end{enumerate}
\end{theorem}

\begin{proof}
\itemref{spatial-cond-1}$\Rightarrow$\itemref{spatial-cond-2}.
Let $F$ be a Scott open filter. Since $j$ is order preserving and commutes with
finite meets, $j^{-1}(F)$ is a filter. If $S\subseteq L$ is directed and
$\bigvee S\in j^{-1}(F)$, then
$\bigvee j[S]=j(\bigvee S)\in F$ because $j$ is Scott continuous.
Since $F$ is Scott open, $j(s)\in F$ for some $s\in S$. Thus,
$s\in j^{-1}(F)$, and so $j^{-1}(F)$ is Scott open.

\itemref{spatial-cond-2}$\Rightarrow$\itemref{spatial-cond-3}. This is immediate since every
completely prime filter is Scott open.

\itemref{spatial-cond-3}$\Rightarrow$\itemref{spatial-cond-1}.
Let $S\subseteq L$ be directed. Since $j$ is order preserving,
$\bigvee j[S]\leq j(\bigvee S)$. If the reverse inequality fails, then
spatiality of $L$ yields a completely prime filter $P$ such that
$j(\bigvee S)\in P$ but $\bigvee j[S]\notin P$. Thus,
$\bigvee S\in j^{-1}(P)$. Since $j^{-1}(P)$ is Scott open, there is
$s\in S$ such that $j(s)\in P$, contradicting
$\bigvee j[S]\notin P$. Therefore,
$j(\bigvee S)=\bigvee j[S]$.
\end{proof}

The spatiality assumption in the above theorem is essential as the next example shows.

\begin{example}
    Let $A$ and $B$ be complete atomless boolean algebras, and let $L$ be their ordinal sum, obtained by identifying $1_A$ with $0_B$; see \cref{fig:example-1}. Then
    \[
        a^{**}=
        \begin{cases}
            a, & a\in A\setminus\{1_A\},\\
            1_B, & a\in B.
        \end{cases}
    \]
    Let $I$ be a maximal ideal of $A$. 
    Since $A$ is atomless, $\bigvee I = 1_A$. 
    Therefore,
        $(\bigvee I)^{**}=1_B$
        while
        $\bigvee\{a^{**} \mid a \in I\} =\bigvee I=1_A$,
    so double negation is not Scott continuous. On the other hand, $L$ has no completely prime filters. Thus, condition \itemref{spatial-cond-3} of 
    \cref{thm:spatial}
    holds vacuously.

    \begin{figure}[H]
        \centering
        \begin{tikzpicture}[
            block/.style={draw, thick, ellipse, minimum width=2.5cm, minimum height=3cm,
            scale=.75},
            dot/.style={circle, fill, inner sep=1.5pt},
            every node/.style={font=\small}
        ]
            \node[block] (A) {$A$};
            \node[block, above=0pt of A] (B) {$B$};

            \node[dot, label=below:{$0_A$}] at (A.south) {};
            \node[dot, label={[xshift=5mm]right:$1_A=0_B$}] at (A.north) {};
            \node[dot, label=above:{$1_B$}] at (B.north) {};
        \end{tikzpicture}
        \caption{The ordinal sum of $A$ and $B$}
        \label{fig:example-1}
    \end{figure}
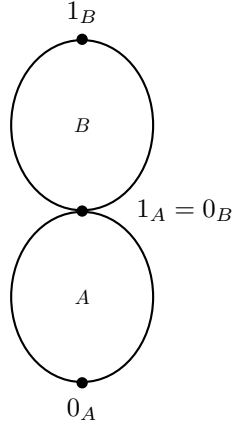
\end{example}

We now recall the Hofmann-Mislove Theorem (see, e.g., \cite[Thm.~II-1.20]{GH+03}).

\begin{theorem*}[Hofmann--Mislove Theorem]
    \label{HM}
    Let $X$ be a sober space and $\Omega(X)$ the frame of opens of $X$. The poset of Scott open filters of $\Omega(X)$ ordered by inclusion is isomorphic to the poset of compact saturated subsets of $X$ ordered by reverse inclusion. The isomorphism sends a Scott open filter $\mathcal F$ to the compact saturated set $\bigcap \mathcal F$, and its inverse sends a compact saturated set $K$ to the Scott open filter $\{U \in \Omega(X) \mid K \subseteq U\}$.
\end{theorem*}

In particular, a Scott open filter is proper iff its corresponding compact saturated set is nonempty. As is common, for $x \in X$, we denote the completely prime filter $\{ U \in \Omega(X) \mid x \in U \}$ by $\mathcal F_x$. The following lemma will be used throughout.

\begin{lemma}\label{lem: key}
    Let $X$ be a sober $T_1$-space. If $j$ is a Scott continuous nucleus on $\Omega(X)$, then $j^{-1}(\mathcal F_x) = \mathcal F_x$ for each $x \notin j(\varnothing)$.
\end{lemma}

\begin{proof}
    Since $\mathcal F_x$ is Scott open, $\mathcal G_x:=j^{-1}(\mathcal F_x)$ is Scott open by 
    \cref{thm:spatial}. We have $$U \in \mathcal F_x \Longrightarrow x \in U \Longrightarrow x \in j(U) \Longrightarrow j(U) \in \mathcal F_x \Longrightarrow U \in \mathcal G_x.$$ Therefore, $\mathcal F_x \subseteq \mathcal G_x$, and so $$G_x := \bigcap \mathcal G_x \subseteq \bigcap \mathcal F_x =: F_x.$$ 
    Since $X$ is a $T_1$-space, $F_x = \{x\}$. Moreover, since $x\notin j(\varnothing)$, we see that $j(\varnothing) \notin \mathcal F_x$, so $\varnothing \notin \mathcal G_x$. Thus,  $\mathcal G_x$ is proper, and hence $G_x$ is a nonempty compact saturated set 
    by the Hofmann--Mislove Theorem. Consequently, $\varnothing \ne G_x \subseteq F_x =\{x\}$, yielding that $G_x = F_x = \{x\}$. Since both $\mathcal F_x$ and $\mathcal G_x$ are Scott open filters, applying the Hofmann--Mislove Theorem again, we conclude that $\mathcal G_x = \mathcal F_x$.
\end{proof}

We now use the above lemma to prove our first main result.

\begin{theorem} \label{thm-new-main}
    Let $X$ be a sober $T_1$-space. The following are equivalent.
    \begin{enumerate}[cref=theorem]
        \item The double negation nucleus on $\Omega(X)$ is Scott continuous;\label{thm-new-main-1}
        \item The double negation nucleus on $\Omega(X)$ is the identity;\label{thm-new-main-2}
        \item $X$ is discrete.\label{thm-new-main-3}
    \end{enumerate}
\end{theorem}
\begin{proof}
    The implications \itemref{thm-new-main-3}$\Rightarrow$\itemref{thm-new-main-2}$\Rightarrow$\itemref{thm-new-main-1} are obvious.
    To see that \itemref{thm-new-main-1}$\Rightarrow$\itemref{thm-new-main-3}, let the double negation nucleus on $\Omega(X)$ be Scott continuous. It suffices to show that every singleton of $X$ is open. Let $x \in X$. Since $\varnothing^{**}=\varnothing$, \cref{lem: key} applies, by which $\mathcal F_x = \mathcal G_x$.
    Thus, $x \in U$ iff $x \in U^{**} = \intr \cl U$ for each $U \in \Omega(X)$. Let $U=X\setminus\{x\}$, which is open because $X$ is $T_1$. Since $x\notin U$, we have $x\notin\intr\cl U$. Therefore, $\cl U\neq X$, and hence $\cl U=U$. Thus, $\{x\}$ is open.
\end{proof}

Since a compact $T_1$-space is discrete iff it is finite, \cref{thm-new-main} immediately yields:

\begin{corollary}\label{cor: compact case}
    If $X$ is a compact sober $T_1$-space, then the double negation nucleus on $\Omega(X)$ is Scott continuous iff $X$ is finite.
\end{corollary}

We next show that neither the sober nor $T_1$ assumption on $X$ can be dropped from \cref{thm-new-main}.

\begin{example}\leavevmode\label{cofinite-example} 
    \begin{enumerate}[cref=example]
        \item Equip an infinite set $X$ with the cofinite topology. Then $X$ is a non-discrete $T_1$-space. On the other hand, since each nonempty open set of $X$ is dense, $U^{**} = \intr \cl U = X$ for each $U \in \Omega(X) \setminus \{\varnothing\}$. Therefore, $\mathfrak B L = \{\varnothing,X\}$. Because the sublocale of fixpoints of the double negation nucleus is finite, it must be Scott continuous by \cref{lem:finite-scts} below.\label{cofinite-example-1}
        \item
        Let $Y$ be the soberification of $X$, so $Y$ is sober (but not $T_1$).
        Since 
        $\Omega(Y)\cong\Omega(X)$ (see, e.g., \cite[Rem.~6.3.1]{PP12}), 
        double negation is Scott continuous on $\Omega(Y)$ by \itemref{cofinite-example-1}. On the other hand, $Y$ is not discrete. 
    \end{enumerate}
\end{example}

\begin{lemma}
    Let $j$ be a nucleus on a frame $L$. 
    \begin{enumerate}[cref=lemma]
        \item $j$ is Scott continuous iff $\bigvee S \in jL$ for each directed $S \subseteq jL$. \label{lem:scts-sublocale}
        \item If $jL$ is finite, then $j$ is Scott continuous. In particular,  every nucleus on a finite frame is Scott continuous. \label{lem:finite-scts}
    \end{enumerate}
    
\end{lemma}
\begin{proof}
    \itemref{lem:scts-sublocale}. This result is mentioned without proof in \cite[p.~649]{Ban88}. For the reader's convenience, we give a short proof. If $j$ is Scott continuous and $S\subseteq jL$ is directed, then
$j(\bigvee S)=\bigvee j[S]=\bigvee S$, so $\bigvee S\in jL$.
Conversely, suppose that $jL$ is closed under directed joins, and let
$S\subseteq L$ be directed. Then $j[S]$ is directed and hence
$\bigvee j[S]\in jL$. Since $\bigvee S\leq\bigvee j[S]$, we have
$j(\bigvee S)\leq j(\bigvee j[S]) = \bigvee j[S]$. Since the converse inequality always holds,
we conclude that $j(\bigvee S)=\bigvee j[S]$.

    \itemref{lem:finite-scts}. Suppose $jL$ is finite and $S \subseteq jL$ is directed. 
    Then $S$ is finite, and since $S$ is directed, it must have a largest element, say $s$. 
    Thus, $j(\bigvee S) = j(s) = s$, so $\bigvee S = s \in j L$, 
    and hence $j$ is Scott continuous by \itemref{lem:scts-sublocale}. Now, if $L$ is finite, then $jL$ is finite for each nucleus $j$. Consequently, each nucleus on a finite frame is Scott continuous.
\end{proof}

It is a consequence of  
\cref{lem:finite-scts} that double negation on every chain is Scott continuous since if $L$ is a chain, then $\mathfrak B L = \{0,1\}$. However, there exist nuclei on chains that are not Scott continuous. The next result provides a characterization of these.

\begin{theorem}
Let $L$ be a chain and $j$ a nucleus on $L$. 
    \begin{enumerate}[cref=theorem]
        \item The following are equivalent.\label{chain-1}
    \begin{enumerate}
        \item $j$ is Scott continuous.\label{chain-cond-1}
        \item If $a \neq 0$ and 
        $a = \bigvee\{b \in L \mid b < a\}$,
        then $j(a) = \bigvee\{j(b) \mid b < a\}$.\label{chain-cond-2}
        \item $jL$ is closed under arbitrary nonempty joins. 
        \label{chain-cond-3}
    \end{enumerate}
    \item If $j$ is a dense nucleus on $L$, then $j$ is Scott continuous iff $jL$ is a subframe of $L$.\label{chain-2}
    \end{enumerate}
\end{theorem}

\begin{proof}
\itemref{chain-1}. Since every nonempty subset of a chain is directed,
\itemref{chain-cond-1} and \itemref{chain-cond-3} are equivalent by
\cref{lem:scts-sublocale}. The implication \itemref{chain-cond-1}$\Rightarrow$\itemref{chain-cond-2} is immediate
since $\{b\in L\mid b<a\}$ is always directed in a chain. For \itemref{chain-cond-2}$\Rightarrow$\itemref{chain-cond-3}, let $S\subseteq jL$ be nonempty and let $a=\bigvee S$. If $a\in S$, then
$a\in jL$. Otherwise, $s<a$ for each $s\in S$, so
$a=\bigvee\{b\in L\mid b<a\}$. By \itemref{chain-cond-2},
$j(a)=\bigvee\{j(b)\mid b<a\}$.
Since $L$ is a chain and $a = \bigvee S$, for each $b<a$, there is $s\in S$ with $b\leq s$. Therefore, 
$j(b)\leq j(s)=s$. Thus, $$j(a)=\bigvee\{j(b)\mid b<a\} \le \bigvee S = a,$$ so 
$j(a)=a$, and hence $a\in jL$.

\itemref{chain-2}. This is immediate from \itemref{chain-1}.
\end{proof}

\begin{remark}[{\bf Connection to Hausdorff frames}]
\label{example:hausdorff}

Since every Hausdorff space is both sober and $T_1$, it is immediate from \cref{thm-new-main} that if $X$ is Hausdorff, then the double negation nucleus on $\Omega(X)$ is Scott continuous iff $X$ is discrete.
It is worth noting that this result does not extend to non-spatial Hausdorff frames. 
Recall (see, e.g., \cite[p.~44]{PP21}) that a frame $L$ is \emph{Hausdorff} provided 
\[
    a = \bigvee\{b \in L \mid b^* \nleq a\}
\]
for each $a \in L \setminus \{1\}$. 
    We show that there exist non-boolean Hausdorff frames in which double negation is Scott continuous. Consider the following instance of the ``slightly peculiar example'' of \cite[pp.~45--46]{PP21}: 
    
    Let $B$ be an atomless boolean frame and let
    \[
        L=\{(a,b)\in B^2\mid a\leq b\}.
    \]
    Clearly $L$ is not boolean, but $L$ is Hausdorff since $B$ is Hausdorff and has no maximal elements, thus \cite[Lem.~III-3.5.1]{PP21} applies. 
    For each $(a,b) \in L$, 
    we have
    \[
        (a,b)^* 
        = \bigvee\{(c,d) \in L \mid a \wedge c =  0 = b \wedge d\} 
        = \bigvee\{(d,d) \in L \mid b \wedge d = 0\}
        = (b^*,b^*).
    \]
    Therefore,
    $(a,b)^{**}=(b^{**},b^{**})=(b,b)$, 
    so 
    \[
    \left(\bigvee_{i\in I}(a_i,b_i)\right)^{**} = \left(\bigvee_{i\in I} a_i,\bigvee_{i\in I} b_i\right)^{**} = \left(\bigvee_{i\in I} b_i,\bigvee_{i\in I} b_i\right) = \left(\bigvee_{i\in I} b_i^{**},\bigvee_{i\in I} b_i^{**}\right) = \bigvee_{i\in I} (a_i,b_i)^{**},
    \]
    and hence the double negation nucleus preserves arbitrary joins. In particular, it is Scott continuous.
\end{remark}

\section{Connection to the existing literature}\label{sec: existing lit}

The topological space $X$ in \cref{cofinite-example}\itemref{cofinite-example-1} is $T_1$, so its frame of opens is subfit (see, e.g., \cite[p.~73]{PP12}). It therefore rules out a pointfree generalization of \cref{thm-new-main} to subfit frames. The frame in \cref{example:hausdorff} provides an analogous counterexample for Hausdorff frames. The situation improves for fit frames since in this case all Scott continuous nuclei are closed by a result of Sexton and Simmons \cite[Thm.~10.6]{SS06}. In \cref{thm:scts=closed} we observe that the same holds for the frame of opens of a sober $T_1$-space. Note that this doesn't follow from the result of Sexton and Simmons since the frame of opens of a sober $T_1$-space need not be fit (see, e.g., \cite[Ex.~5.3]{SS06}). This result is then strengthened in \cref{thm: 2nd main}, where we prove that, among $T_0$-spaces, the sober $T_1$-spaces are precisely those for which every Scott continuous nucleus on the frame of opens is closed, thus yielding a pointfree characterization of sober $T_1$-spaces.

Recall that a nucleus on a frame $L$ is \emph{closed} if it is of the form
\[
    \mathfrak c_u(a)=a\vee u
\]
for some $u\in L$. 
Every closed nucleus preserves nonempty joins since for each nonempty $S\subseteq L$,
\[
    \mathfrak c_u\!\left(\bigvee S\right)
    =u\vee\bigvee S
    =\bigvee\{u\vee s\mid s\in S\}
    =\bigvee\{\mathfrak c_u(s)\mid s\in S\}.
\]
Thus, every closed nucleus is Scott continuous, as noted in \cite[p.~6]{Esc99}. Therefore, the canonical embedding $a\mapsto\mathfrak c_a$ of $L$ into its assembly $\mathrm N L$ factors through $\mathrm{SN} L$. In particular, we always have
\begin{equation} \label{dagger}
    L\cong\{\text{closed nuclei on }L\}\subseteq\mathrm{SN} L\subseteq\mathrm N L. \tag{$\dagger$}
\end{equation}
If every Scott continuous nucleus is closed, then the first inclusion above is an equality and $\mathrm{SN} L\cong L$. We briefly record when the second inclusion is an equality. For this we use the following 
result of Simmons, where we recall that a nucleus on $L$ is \emph{open} if it is of the form
\[
    \mathfrak o_u(a)=a\to u
\]
for some $u\in L$.

\begin{proposition}[{\cite[Lem.~7(i)]{Sim78}}]
    Every nucleus on $L$ is the join of the nuclei $\mathfrak o_a\wedge \mathfrak c_b$ beneath it. \label{prop:joins-locally-closed}
\end{proposition}

\begin{theorem}\label{thm:SNL=NL}
    Let $L$ be a frame. Then $\mathrm{SN} L = \mathrm N L$ iff each open nucleus is Scott continuous.
\end{theorem}

\begin{proof}
The 
forward implication is immediate. Conversely, suppose that every open nucleus is Scott continuous. By \cref{lem:SNL-subframe}, $\mathrm{SN} L$ 
is a subframe of $\mathrm N L$. Since closed nuclei are Scott continuous, each nucleus of the form $\mathfrak o_a\wedge\mathfrak c_b$ is Scott continuous, and hence so is every join of such nuclei. Now apply \cref{prop:joins-locally-closed} 
to conclude that every nucleus on $L$ is Scott continuous.
\end{proof}

We now turn our attention to the first inclusion in \eqref{dagger}. Recall (see, e.g., \cite[p.~74]{PP12}) that a frame $L$ is \emph{fit} if $a\nleq b$ implies that there is $c\in L$ such that $a\vee c=1$ and $c\to b\nleq b$. For fit frames, Sexton and Simmons proved that the first inclusion in \eqref{dagger} is always an equality:
\begin{theorem}[{\cite[Thm.~10.6]{SS06}}]
    \label{thm:SS}
    If $L$ is a fit frame, then each Scott continuous nucleus on $L$ is closed.
\end{theorem}

\begin{remark}\leavevmode
    \cref{thm:SS} fails if fit is replaced by Hausdorff since
    double negation on the Hausdorff frame considered in \cref{example:hausdorff} is Scott continuous, but not closed. To see the latter, observe that a nucleus $j$ is closed iff $j = \mathfrak c_{j(0)}$. Therefore, a closed dense nucleus must be the identity. 
\end{remark}

As a consequence of \cref{thm:SS}, we obtain:

\begin{corollary} \label{cor:SS}
    If $L$ is a fit frame, then double negation is Scott continuous iff $L$ is boolean.
\end{corollary}

\begin{proof}
    If $L$ is boolean, then double negation is the identity, and hence it is Scott continuous. Conversely, suppose double negation is Scott continuous. By \cref{thm:SS}, it is closed, and hence equal to $\mathfrak c_{0^{**}}=\mathfrak c_0$, which is the identity. Thus, double negation is the identity on $L$, so $L=\mathfrak B L$, and hence $L$ is boolean.
\end{proof}

We next 
show that every Scott continuous nucleus on the frame of opens of a sober $T_1$-space is closed.

\begin{proposition}\label{thm:scts=closed}
    If $X$ is a sober $T_1$-space, then every Scott continuous nucleus $j$ on $\Omega(X)$ is closed. 
\end{proposition}

\begin{proof}
    Since $\mathfrak c_{j(\varnothing)}\leq j$ always holds, it is enough to show the reverse inequality. Let $U\in\Omega(X)$ and $x \in j(U)$. If $x \notin j(\varnothing)$, then  
    $j^{-1}(\mathcal F_x)=\mathcal F_x$ by \cref{lem: key}. Thus, since $U \in j^{-1}(\mathcal F_x)$, we must have 
    $x\in U$, so $j(U)\subseteq U\cup j(\varnothing)=\mathfrak c_{j(\varnothing)}(U)$. 
\end{proof}

\begin{remark}
As we pointed out at the beginning of \cref{sec: existing lit}, the above proposition is not a consequence of \cref{thm:SS}. 
It can however be derived from existing results in the literature. Indeed, by \cite[Thm.~3.1.4]{Esc98}, 
Scott continuous sublocales of spatial frames are spatial. Consequently, each Scott continuous sublocale of $\Omega(X)$ is induced by a subspace of $X$ (see, e.g., \cite[Prop.~VI-2.2.1]{PP12}), which may be assumed to be sober since  $X$ is sober. Hofmann and Lawson \cite[Prop.~3.3]{HL84} showed that if $f : X \to Y$ is a continuous map between sober spaces, then the right adjoint of $\Omega(f) : \Omega(Y) \to \Omega(X)$ preserves directed joins iff $f$ is a \emph{proper map}; that is,
\begin{enumerate}[(i)]
    \item $f^{-1}(K)$ is compact for each compact saturated $K \subseteq Y$;
    \item the downset in the specialization order of $f(C)$ is closed for each closed $C \subseteq X$. \label{closed-downset}
\end{enumerate}
Consequently, the Scott continuous sublocales of $\Omega(X)$ correspond to sober subspaces $Y$ of $X$ whose subspace embedding is proper (see, e.g., \cite[Cor.~3.2.4]{Esc98}). If $X$ is $T_1$, the specialization order becomes equality, so condition \ref{closed-downset} forces $Y$ to be a closed subset of $X$. Thus, if $X$ is sober and $T_1$, then each Scott continuous nucleus on $\Omega(X)$ is closed, yielding \cref{thm:scts=closed}. We find our proof above much shorter. 
\end{remark}

We are ready for our second main result, that among $T_0$-spaces, the property that every Scott continuous nucleus on $\Omega(X)$ is closed 
singles out
sober $T_1$-spaces.
For this, recall (see, e.g., \cite[Rem.~III.6.2.2]{PP12}) that, for each $U\in\Omega(X)$, the closed sublocale
$\mathfrak c_U\Omega(X)=[U,X]$ is 
isomorphic to $\Omega(X\setminus U)$. 

\begin{theorem}\label{thm: 2nd main}
    Let $X$ be a $T_0$-space. Then $X$ is sober and $T_1$ iff every Scott continuous nucleus on $\Omega(X)$ is closed.
\end{theorem}

\begin{proof}
    ($\Rightarrow$). This follows from \cref{thm:scts=closed}.

    ($\Leftarrow$). 
    Suppose every Scott continuous nucleus on $\Omega(X)$ is closed. Let $Y$ be the soberification of $X$. If $Y$ is $T_1$, then $X\cong Y$ (see, e.g., \cite[Note~VI-2.3.2]{PP12}), and hence $X$ is sober and $T_1$. Since $\Omega(Y)\cong\Omega(X)$, the hypothesis also holds for $Y$. It is therefore enough to prove that every sober space satisfying the hypothesis is $T_1$. Thus, we assume that $X$ is sober and prove that $X$ is $T_1$. 
    Let $x \in X$ and let $U = X \setminus \cl\{x\}$. 
    
    \begin{claim} \label{claim}
    $\mathfrak b_U \Omega(X) = \{U,X\}$.
    \end{claim}
    \begin{proof}[Proof of the claim]\renewcommand{\qedsymbol}{$\diamond$}
        Since $U$ is meet-irreducible in $\Omega(X)$, the claim follows from the observations in \cite[p.~42]{PP12}. However, it is easy to give a direct proof as well. 
        Let $V \in \Omega(X)$. If $V\subseteq U$, then $V\to U=X$, so
    $\mathfrak b_U(V)=U$. Let $V\nsubseteq U$. We show that $\mathfrak b_U(V) = X$. From $V\nsubseteq U$ it follows that $V\cap\cl\{x\}\neq\varnothing$,
    so $x\in V$. Now let $W \in \Omega(X)$ with 
    $V\cap W\subseteq U$. Since $x \in V$ and $x\notin U$, we have $x \notin W$. Thus, $W \cap \cl \{x\} = \varnothing$,
    so $W \subseteq U$. 
    Since $V\to U = \bigcup \{ W \in \Omega(X) \mid V \cap W \subseteq U \}$, we conclude that
    $V\to U=U$, and
    hence $\mathfrak b_U(V)=X$.
    \end{proof}
    
    By \cref{claim}, $\mathfrak b_U \Omega(X)$ is finite, so $\mathfrak b_U$ is Scott continuous by \cref{lem:finite-scts}. Therefore, $\mathfrak b_U$ is closed
    by assumption. Since $\mathfrak b_U(\varnothing)=U$, we have $\mathfrak b_U=\mathfrak c_U$.
    Thus, $\mathfrak c_U \Omega(X)=\{U,X\}$. But $\mathfrak c_U \Omega(X)\cong\Omega(\cl\{x\})$, so the subspace topology on 
    $\cl\{x\}$ is trivial. Since $X$ is $T_0$, 
    we conclude that $\cl\{x\} = \{x\}$. Thus, every
    singleton is closed, and hence $X$ is~$T_1$.
\end{proof}

The above theorem suggests regarding frames in which every Scott continuous nucleus is closed  as pointfree analogs of sober $T_1$-spaces. 

\section{Scott continuity for boolean nuclei}

We now extend \cref{thm-new-main} 
to arbitrary boolean nuclei, yielding our third main result. This requires the following:

\begin{lemma} \label{lem:bU-restriction}
    Let $L$ be a frame and $a\in L$. Then $\mathfrak b_a$ is Scott continuous 
    iff its restriction to $\mathfrak c_aL$ is Scott continuous.
\end{lemma}

\begin{proof}
    ($\Rightarrow$). Suppose that $\mathfrak b_a$ is Scott continuous. 
    Let $S\subseteq\mathfrak c_aL$ be directed. Since $\mathfrak c_aL=[a,1]$, nonempty joins in $\mathfrak c_aL$ agree with those in $L$. Therefore,
    \[
        \mathfrak b_a\!\left(\bigvee_{\mathfrak c_aL}S\right)
        =\mathfrak b_a\!\left(\bigvee_LS\right)
        =\bigvee_L\mathfrak b_a[S]
        =\bigvee_{\mathfrak c_aL}\mathfrak b_a[S].
    \]
    Thus, the restriction of $\mathfrak b_a$ to $\mathfrak c_aL$ is Scott continuous.

    ($\Leftarrow$). 
    Suppose that the restriction of $\mathfrak b_a$ to $\mathfrak c_aL$ is Scott continuous. Let $S\subseteq L$ be directed. Then $\mathfrak c_a[S]\subseteq\mathfrak c_aL$ is directed. Since $\mathfrak c_a$ preserves nonempty joins and $\mathfrak b_a\circ\mathfrak c_a=\mathfrak b_a$ (the latter is easy to verify, but also follows from \cite[Lem.~3.6]{Mac81}),
    we have
    \[
        \mathfrak b_a\!\left(\bigvee_LS\right)
        =\mathfrak b_a \mathfrak c_a\!\left(\bigvee_LS\right)
        =\mathfrak b_a\!\left(\bigvee_{\mathfrak c_aL}\mathfrak c_a[S]\right)
        =\bigvee_{\mathfrak c_aL}\mathfrak b_a[\mathfrak c_a[S]]
        =\bigvee_L\mathfrak b_a[S].
    \]
    Thus, $\mathfrak b_a$ is Scott continuous.
\end{proof}

\begin{theorem} \label{cor:bU-scts-iff-discrete}
    Let $X$ be a sober $T_1$-space. For $U\in\Omega(X)$, 
    the following are equivalent.
    \begin{enumerate}[cref=theorems]
        \item $\mathfrak b_U$ is Scott continuous; \label{bu-scts-1}
        \item $\mathfrak b_U=\mathfrak c_U$; \label{bu-scts-2}
        \item $X\setminus U$ is discrete. \label{bu-scts-3}
    \end{enumerate}
\end{theorem}

\begin{proof}
    The bottom element of $\mathfrak c_U\Omega(X)$ is $U$, so the pseudocomplement of $V\in\mathfrak c_U\Omega(X)$ is $V\to U$.
    Thus, double negation on $\mathfrak c_U\Omega(X)$ is precisely the restriction $\mathfrak b_U : \mathfrak c_U\Omega(X) \to \mathfrak c_U\Omega(X)$. 
    By \cref{lem:bU-restriction}, \itemref{bu-scts-1} holds iff double negation on $\mathfrak c_U \Omega(X)$ is Scott continuous.
    We show that \itemref{bu-scts-2} holds iff double negation on $\mathfrak c_U\Omega(X)$ is the identity. 
    If $\mathfrak b_U=\mathfrak c_U$, then the restriction of
    $\mathfrak b_U$ to $\mathfrak c_U \Omega(X)$ is the identity. 
    Conversely, if this restriction is the identity, then $\mathfrak b_U(V)=\mathfrak b_U(\mathfrak c_U(V))=\mathfrak c_U(V)$ for each $V\in \Omega(X)$.
    Thus, under the isomorphism $\mathfrak c_U\Omega(X)\cong\Omega(X\setminus U)$, 
    \itemref{bu-scts-1}--\itemref{bu-scts-3} become precisely the three equivalent conditions of \cref{thm-new-main}, and the result follows.
\end{proof}

\begin{corollary}
    Let $X$ be a compact sober $T_1$-space and $U \in \Omega(X)$. Then $\mathfrak b_U$ is Scott continuous iff 
    $X \setminus U$ is finite.
\end{corollary}

\begin{proof}
    By \cref{cor:bU-scts-iff-discrete}, $\mathfrak b_U$ is Scott continuous iff
$X\setminus U$ is discrete. Since $X\setminus U$ is a closed subset of a compact space, it is compact. Thus, $X\setminus U$ is discrete iff $X\setminus U$ is finite.
\end{proof}

\section{Compactness and \texorpdfstring{$j$}{j}-units}

As we saw in \cref{thm:scts=closed},
Scott continuous nuclei on the frame of opens of a sober $T_1$-space are closed. In particular, if double negation is Scott continuous, then it is the identity, forcing the space $X$ to be discrete (see \cref{thm-new-main}). In the compact case, this result can be sharpened to $X$ being finite (see \cref{cor: compact case}).
More generally, we prove that if a frame $L$ has a Scott open filter contained in
$\mathfrak b_a^{-1}(1)=\{b\in L\mid \mathfrak b_a(b)=1\}$, then $\mathfrak b_a$ is Scott continuous iff $\mathfrak b_aL$ is finite.

We recall (see \cite{KM07}) that a \emph{unit} is a compact dense element of a frame $L$. 
In \cite{BB+26}, this notion was generalized to an \emph{S-unit}; that is, an element $a \in L$ such that 
\[ \label{S-unit}\tag{$\ddagger$}
\{b \in L \mid a \text{ is way below }b\}
\]
is a Scott open filter consisting of dense elements.
The next definition provides further generalization.

\begin{definition}
    Let $j \in \mathrm N L$. A \emph{$j$-unit} is a Scott open filter $F \subseteq L$ such that $F \subseteq j^{-1}(1)$.
\end{definition}

In \cite[Lem.~3.4]{Joh85} it was observed that $jL$ is compact iff $j^{-1}(1)$ is a Scott open filter. This together with \cite[Prop.~II-2.1]{GH+03} yields the following:

\begin{lemma}\label{lem:j-unit-iff-jL-compact}
    Let $L$ be a frame and $j \in \mathrm{SN} L$. The following equivalent. 
    \begin{enumerate}[cref=lemma]
        \item $jL$ is compact; \label{j-unit-1}
        \item $j^{-1}(1)$ is a Scott open filter;\label{j-unit-2}
        \item $L$ has a $j$-unit.\label{j-unit-3}
    \end{enumerate}
\end{lemma}

\begin{theorem}
    Let $L$ be a frame and $a\in L$.
    \begin{enumerate}[cref=theorem]
        \item If $L$ has a $\mathfrak b_a$-unit, then $\mathfrak b_a$ is Scott
        continuous iff $\mathfrak b_aL$ is finite. 
        \label{thm:j-unit-scts-iff-finite}
        
        \item If $L$ has an $S$-unit, then the double negation nucleus is
        Scott continuous iff $\mathfrak B L$ is finite.\label{thm:s-unit-scts-iff-finite}
    \end{enumerate}
\end{theorem}

\begin{proof}
     \itemref{thm:j-unit-scts-iff-finite}. Suppose $\mathfrak b_a$ is Scott continuous. Since $L$ has a $\mathfrak b_a$-unit, \cref{lem:j-unit-iff-jL-compact} implies that $\mathfrak b_aL$ is compact. 
     Thus, $\mathfrak b_aL$ is finite because it is a compact boolean frame. The converse follows from \cref{lem:finite-scts}. 

     \itemref{thm:s-unit-scts-iff-finite}. This is the special case $a=0$ of \itemref{thm:j-unit-scts-iff-finite} since 
     every $S$-unit induces a $\mathfrak b_0$-unit via \eqref{S-unit}
     and $\mathfrak b_0L=\mathfrak B L$.
\end{proof}

If $a$ is compact and $j(a)=1$, then ${\uparrow} a$ is a $j$-unit. Consequently, every compact frame has a $j$-unit for every nucleus $j$. Applying this to the previous theorem yields: 

\begin{corollary}
    Let $L$ be a compact frame and $a \in L$.
    \begin{enumerate}[cref=theorem]
        \item $\mathfrak b_a$ is Scott
        continuous iff $\mathfrak b_aL$ is finite. 
        
        \item The double negation nucleus is
        Scott continuous iff $\mathfrak B L$ is finite.
    \end{enumerate}
\end{corollary}

\newcommand{\etalchar}[1]{$^{#1}$}

\end{document}